\documentclass[12pt]{article}

\usepackage[latin1]{inputenc}
\usepackage[T1]{fontenc}

\usepackage{authblk} 

\usepackage{amsmath}
\usepackage{amsfonts}
\usepackage{amssymb}

\usepackage[normalem]{ulem} 

\usepackage{amsmath,xcolor,ulem}
\usepackage{amsfonts}
\usepackage{amssymb,todonotes}
\usepackage{amsthm}
\usepackage{graphicx}
\usepackage{marginnote}

\usepackage[colorlinks=true, allcolors=blue]{hyperref}

\usepackage[top=2.5cm,bottom=2.5cm,left=2.5cm,right=2.5cm]{geometry}
\usepackage{float}
\usepackage[algoruled]{algorithm2e}
\usepackage[font=footnotesize,width=\linewidth]{caption} 

\usepackage{graphicx}
\usepackage{amssymb}
\usepackage{amsthm}
\usepackage{amsmath}
\usepackage{lineno}
\usepackage{hyperref}
\usepackage{verbatim}

\newtheorem{theorem}{Theorem}[section]

\theoremstyle{definition}

\newtheorem{example}[theorem]{Example}

\theoremstyle{remark}
\newtheorem{remark}[theorem]{Remark}

\numberwithin{equation}{section}

\usepackage{graphicx, pdflscape}
\usepackage{amsmath, hyperref}
 \usepackage{amssymb} 
\usepackage{relsize, float, lineno}

 \usepackage{xcolor} 
\usepackage{subcaption}
\allowdisplaybreaks
\newcommand{\diff}{\mathrm{d}}
\DeclareMathOperator{\err}{err}
\DeclareMathOperator{\ROC}{ROC}

\title{A Mathematical Study of an SIS Epidemic Model: Global Asymptotic Stability Analysis and Construction of Nonstandard Numerical Schemes}

\author{Manh Tuan Hoang\footnote{Email(s): \href{mailto:tuanhm16@fe.edu.vn}{tuanhm16@fe.edu.vn}; \href{mailto:hmtuan01121990@gmail.com}{hmtuan01121990@gmail.com}}}

\affil{Department of Mathematics, FPT University, Hoa Lac Hi-Tech Park, \\ Km29 Thang Long Blvd, Hanoi, Viet Nam}
\begin{document}
\maketitle

\begin{abstract}
The aim of this work is to provide a rigorous mathematical analysis of a well-known SIS epidemic model with a saturating contact rate. First, we establish the global asymptotic stability (GAS) of the model's equilibria by employing a suitable Lyapunov function in combination with the Poincar\'e--Bendixson theorem and the Bendixson--Dulac criterion. The resulting GAS results improve upon previous findings for this SIS model and may also be applicable to its extensions with more general saturating contact rates.

Second, we construct families of first- and second-order nonstandard finite difference (NSFD) schemes that preserve the positivity and asymptotic stability properties of the SIS model for arbitrary step sizes. All these schemes are formulated within Mickens' framework; however, compared with their first-order counterparts, the second-order schemes employ a more elaborate construction that combines a weighted nonlocal approximation of the right-hand side with suitably renormalized denominator functions. The weighted nonlocal approximation guarantees dynamic consistency, whereas the denominator functions ensure second-order convergence.

Finally, we conduct numerical experiments to validate the theoretical results and demonstrate the advantages of the proposed  NSFD schemes. The numerical results are in good agreement with the theoretical results. The approach developed in this work is applicable not only to other epidemiological systems but also, more generally, to mathematical models arising in a wide range of real-world applications.
\end{abstract}

\begin{minipage}{0.9\linewidth}
\end{minipage}
\section{Introduction}
Mathematical modeling and analysis in epidemiology have many useful real-world applications \cite{Allen, Brauer, Martcheva}. Building on the foundational works in mathematical epidemiology by Kermack and McKendrick \cite{Kermack1, Kermack2, Kermack3}, mathematical models have been developed to investigate many serious infectious diseases, such as influenza, hepatitis B, and COVID-19 and so on (see, e.g., \cite{Alexander, Casagrandi, Lacarbonara, Nowak, Tchoumi, Wang}).
The study of epidemiological models helps elucidate the mechanisms of disease transmission and spread and informs the development of disease prevention and control measures to protect public health. A recent review \cite{HoangMatthias} surveyed advances in the use of deterministic differential equations for studying infectious diseases, focusing on mathematical modeling, qualitative analysis, numerical methods, and practical applications. 
Recently, we have constructed dynamically consistent numerical schemes for epidemic models of influenza, HBV, and COVID-19 \cite{Hoang1, Hoang2, Hoang3, HoangMatthias1}, which are based on Mickens' nonstandard finite difference (NSFD) methodology \cite{Mickens1, Mickens2, Mickens3, Mickens4, Mickens5}. Furthermore, mathematical modeling and NSFD schemes for epidemiological models of malware and information spreading have been investigated in \cite{DangHoangJCAM, Hoang4, Hoang5}.

In \cite{Lan}, Lan et al. proposed a modified version of the classical SIS model, which used a generalized saturating contact rate function. The mathematical model is represented by
\begin{equation}\label{eq:1}
\begin{split}
\dfrac{\diff S}{\diff t} = F(S, I) := \Lambda - \dfrac{{\beta} S I}{h(N)}- \mu S + \gamma I,\\
\dfrac{\diff I}{\diff t} = G(S, I) := \dfrac{{\beta} S I}{h(N)}-(\mu+\delta+\gamma) I,
\end{split}
\end{equation}
where 
\begin{itemize}
\item $S(t)$ and $I(t)$ denote the numbers of susceptible and infectious individuals at time $t$, respectively;
\item $N(t)=S(t)+I(t)$ denotes the total population size;
\item $\Lambda$ and $\delta$ denote the recruitment and disease-induced death rates, respectively;
\item $\mu$ is the natural death rate, whereas $\gamma^{-1}$ represents the mean time to recovery in the absence of competing mortality;
\item $h(N)=1+bN+\sqrt{1+2bN}$ is an auxiliary function defining the saturating contact rate $C(N)= \frac{bN}{h(N)}$;
\item $\beta$ is the transmission coefficient, and $b>0$ is the saturation parameter.
\end{itemize}
In \cite{Lan}, the dynamical properties of the deterministic model were first analyzed; the model was then extended to a stochastic setting, and the resulting stochastic model was subsequently investigated. Hoang et al. \cite{Hoang6} considered a fractional-order version of \eqref{eq:1} in the context of the Caputo derivative and analyzed the proposed fractional model. In particular, the following results were established in \cite{Lan}:
\begin{itemize}
\item The basic reproduction number of \eqref{eq:1} is computed as 
\begin{equation*}
\mathcal{R}_0 = \dfrac{\Lambda \beta}{\mu(\mu+\delta+\gamma) h\left(\dfrac{\Lambda}{\mu}\right)}.
\end{equation*}
\item The model always admits a unique disease--free equilibrium (DFE), $E_0$, whereas a unique disease-endemic equilibrium (DEE) exists if and only if $\mathcal{R}_0 > 1$.
\item The DFE point is not only locally asymptotically stable but also globally asymptotically stable if $\mathcal{R}_0 < 1$; whereas the DEE point is locally asymptotically stable if $\mathcal{R}_0 > $1.
\end{itemize}

It is noteworthy that the global asymptotic stability (GAS) of the DEE has not yet been established theoretically. Nevertheless, numerical simulations reported in \cite{Lan} as well as well-known stability results for epidemiological models \cite{Allen, Brauer, Martcheva} suggest that the DEE may be not only locally asymptotically stable but also globally asymptotically stable. Moreover, the stability of the DEE characterizes the long-term persistence of the disease and is therefore epidemiologically important for understanding disease dynamics and developing effective control strategies.

Motivated by the above observations, our first objective is to establish the GAS of the DEE point. To this end, we employ the Poincar\'e--Bendixson theorem and the Bendixson--Dulac criterion \cite{Allen, Martcheva} with a suitable Dulac function to establish the GAS of the DEE. As expected, we conclude that the DEE is globally asymptotically stable whenever it exists, namely, when $\mathcal{R}_0>1$. In addition, we provide two alternative proofs of the GAS of the DFE using different approaches. The first employs a suitable Lyapunov function, whereas the second applies the result in \cite{Castillo-Chavez}, which relies on the role of the basic reproduction number $\mathcal{R}_0$ in stability analysis. Both proofs differ from the proof given in \cite{Lan}.

In addition to analyzing the GAS of the two equilibria, our second objective is to construct NSFD schemes for numerical simulation purposes. As is widely recognized, the main advantage of NSFD schemes over standard numerical schemes is their dynamic consistency; that is, they preserve essential dynamical properties of the differential equations for all finite step sizes \cite{Mickens1, Mickens2, Mickens3, Mickens4, Mickens5}. For this reason, Mickens' methodology has been extensively developed for the construction of dynamically consistent NSFD schemes. Comprehensive reviews of NSFD methods can be found in \cite{Patidar1, Patidar2}. More recently, Hoang and Ehrhardt provided a comprehensive introduction to NSFD methods and their applications in \cite{HoangMatthias}. However, a drawback, which may be regarded as a trade-off for dynamic consistency, is that most NSFD schemes are only first-order accurate (see, for example, \cite{Cresson}).
In recent years, the construction of higher-order NSFD schemes for differential equations has attracted considerable research attention (see, e.g., \cite{Alalhareth1, Alalhareth2, Alalhareth3, Conte, DangHoang, Hoang_NumAl, HoangMatthias1, HoangMatthias2, Kojouharov1, Takacs} and references therein).
In earlier studies, a class of second-order NSFD schemes for ODEs with polynomial right-hand sides was introduced in \cite{Chen-Charpentier}, whereas higher-order NSFD schemes based on extrapolation techniques and variable-step-size algorithms were developed for MSEIR and SEIR models in \cite{Martin-Vaquero1, Martin-Vaquero2}. In another work \cite{GParra}, NSFD schemes were combined with Richardson extrapolation to improve the accuracy of numerical solutions for several population models.
In Section \ref{Sec3}, we construct families of first- and second-order nonstandard finite difference (NSFD) schemes that preserve the positivity and asymptotic stability properties of the SIS model for arbitrary step sizes. All these schemes are formulated within Mickens' framework; however, the second-order NSFD schemes follow the approach developed in \cite{HoangMatthias2}.  Compared with their first-order counterparts, the second-order schemes employ a more elaborate construction that combines a weighted nonlocal approximation of the right-hand side with suitably renormalized denominator functions. The weighted nonlocal approximation guarantees dynamic consistency, whereas the denominator functions ensure second-order convergence.

Lastly, we conduct several numerical experiments to support and illustrate the theoretical findings. As expected, the numerical results are in good agreement with the theoretical analysis.

The rest of the paper is organized as follows:\\
A complete global asymptotic stability analysis of the SIS model is presented in Section \ref{Sec2}. The construction of first- and second-order NSFD schemes is presented in Section \ref{Sec3}. Numerical simulations are conducted in Section \ref{Sec4}. Finally, the last section provides concluding remarks and discussions.
\section{Complete Global Asymptotic Stability of the SIS Model}\label{Sec2}
The main aim of this section is to  analyze the GAS of the DEE point of \eqref{eq:1}. We recall from the analysis in \cite{Lan} that the DFE point is given by 
\begin{equation*}
E_0 = \left(S_0,\,I_0\right) = \left(\dfrac{\Lambda}{\mu},\,0\right),
\end{equation*}
whereas the DEE point $E_* = (S_*,\,I_*)$ is determined by
\begin{equation*}
S_* = \dfrac{\Lambda}{\mu} - \dfrac{(\delta + \mu)I_*}{\mu}, \quad I_* = \dfrac{\Lambda - \mu N_*}{\delta},
\end{equation*}
where $N_*$ is the unique positive solution belonging to $\left(\dfrac{\Lambda}{\mu + \delta},\,\dfrac{\Lambda}{\mu}\right)$ of the equation (see \cite[Theorem 2]{Lan})
\begin{equation*}
F(N)=\beta\left[N-\frac{\mu}{\delta}\left(\frac{A}{\mu}-N\right)\right] - (\mu+\delta+\gamma) h(N)  = 0.
\end{equation*}
Note that \eqref{eq:1} admits the set $\mathbb{R}_+^2$ as a positively invariant set. Hence, it follows from \eqref{eq:1} that
\begin{equation*}
\Lambda - (\mu + \delta)N \leq \dfrac{\diff N}{\diff t} = \Lambda - \mu N - \delta I \leq \Lambda - \mu N.
\end{equation*}
Using a comparison theorem for ODEs \cite{McNabb} yields
\begin{equation*}
\left(N(0) - \dfrac{\Lambda}{\mu + \delta}\right)e^{-(\mu + \delta)t} + \dfrac{\Lambda}{\mu + \delta} \leq N(t) \leq \left(N(0) - \dfrac{\Lambda}{\mu}\right)e^{-\mu t} + \dfrac{\Lambda}{\mu}, 
\end{equation*}
which implies that
\begin{equation*}
\begin{split}
&\liminf_{t \to \infty}N(t) \geq \dfrac{\Lambda}{\mu + \delta},\\
&\limsup_{t \to \infty}N(t) \leq \dfrac{\Lambda}{\mu}.
\end{split}
\end{equation*}
Thus, it is sufficient to investigate the dynamics of \eqref{eq:1} over a feasible region defined by
\begin{equation*}
\Omega = \left\{(S, I)|S, I \geq 0,\, \dfrac{\Lambda}{\mu + \delta} \leq S + I \leq \dfrac{\Lambda}{\mu}\right\}.
\end{equation*}
To investigate the GAS of the DEE point, we need the following auxiliary result.
\begin{theorem}\label{Theorem1}
If $S(0) \geq 0$ and $I(0) > 0$, then $S(t) > 0$ and $I(t) > 0$ for all $t > 0$.
\end{theorem}
\begin{proof}
First, the second equation of \eqref{eq:1} implies that
\begin{equation*}
I(t) = I(0)e^{\mathlarger{\mathlarger{\int}}_0^t\left[\dfrac{{\beta} S(\tau)}{h(N(\tau))}-(\mu+\delta+\gamma)\right]\diff \tau}.
\end{equation*}
Consequently, $I(t) > 0$ for $t > 0$ whenever $I(0) > 0$, and $I(t) = 0$ for $t > 0$ if $I(0) = 0$.

Assume that $S(0) > 0$. We will show that $S(t) > 0$ for $t > 0$. Suppose, to the contrary, that there exists a first time $t_0 > 0$ such that $S(t_0)=0$. We define $t_* = \min\{t|S(t) = 0\}$. Then, 
\begin{equation}\label{eq:3}
S(t) > 0, \quad \forall 0 < t < t_*.
\end{equation}
At $t = t_*$, we have
\begin{equation}\label{eq:4}
\dfrac{\diff S}{\diff t}\bigg|_{t = t_*} = \Lambda - \dfrac{{\beta} S(t_*) I(t_*)}{h(N(t_*))}- \mu S(t_*) + \gamma I(t_*) = \Lambda + \gamma I(t_*) > 0.
\end{equation}
From the continuity of $\dfrac{\diff S}{\diff t}$, there exists $\epsilon > 0$ such that
\begin{equation*}
\dfrac{\diff S(t)}{\diff t} > 0, \quad t \in (t_* - \epsilon,\,t_* +\epsilon),
\end{equation*}
which implies that $S(t) < S(t_*) = 0$ for $t_* - \epsilon < t < t_*$. This contradicts \eqref{eq:3}. 

If $S(0) = 0$, then it follows from \eqref{eq:4} that there exists $t_0 > 0$ such that $S(t_0) > 0$. By the similar arguments, we conclude that $S(t) > 0$ for $t > 0$. Therefore, the desired conclusion follows. The proof is complete.
\end{proof}
The following theorem can be considered as the main result of this section.
\begin{theorem}[Global asymptotic stability of the DEE point]\label{Theorem2}
The DEE point of the SIS model \eqref{eq:1} is globally asymptotically stable whenever it exists and $I(0) > 0$.
\end{theorem}
\begin{proof}
First, we rewrite \eqref{eq:1} in the form
\begin{equation}\label{eq:1new}
\begin{split}
&\dfrac{\diff I}{\diff t} = f(I, N) := \dfrac{\beta(N - I)I}{h(N)} - (\mu+\delta+\gamma) I,\\
&\dfrac{\diff N}{\diff t} = g(I, N) := \Lambda - \mu N - \delta I.
\end{split}
\end{equation}
Then, the DEE point is transformed to $\widehat{E}_* = (I_*,\,N_*)$.

As a direct consequence of Theorem \ref{Theorem1}, the interior of $\mathbb{R}_2^+$ forms a positively invariant set of \eqref{eq:1new}. Let us consider a candidate Dulac function given by
\begin{equation*}
D(I, N) = \dfrac{1}{I}, \quad I, N > 0.
\end{equation*}
Then, the function $D$ satisfies
\begin{equation*}
\dfrac{\partial(Df)}{\partial I} + \dfrac{\partial(Dg)}{\partial N} = \dfrac{\partial}{\partial I}\left[\dfrac{\beta(N - I)}{h(N)} - (\mu + \delta + \gamma)\right] + \dfrac{\partial}{\partial N}\left(\dfrac{\Lambda}{I} - \dfrac{\mu N}{I} - \delta\right) = -\dfrac{\beta}{h(N)} - \dfrac{\mu}{I} < 0
\end{equation*}
for all $I, N > 0$. Using the Dulac--Bendixson criterion (\cite[Theorem 3.6]{Martcheva}) implies that the system \eqref{eq:1new} has no periodic orbits or graphics in the first quadrant.

Now, by repeating the arguments used in the proof of \cite[Theorem 3.8]{Martcheva}, we conclude that the DFE point does not belong to the omega limit set of $\left(I(0),\, N(0)\right)$. Therefore, $\lim _{t \to \infty}\left(I(t),\, N(t)\right) = \widehat{E}_*$. Combining this with the local asymptotic stability of $\widehat{E}_*$, we obtain its GAS, and the proof is complete.
\end{proof}

Before concluding this section, we provide two alternative proofs of the GAS of the DFE using different approaches. The first employs a suitable Lyapunov function, whereas the second applies the result in \cite{Castillo-Chavez}, which relies on the role of the basic reproduction number $\mathcal{R}_0$ in stability analysis. Both proofs differ from the proof given in \cite{Lan}.
\begin{theorem}[GAS analysis of the DFE point via a Lyapunov function]\label{Theorem3}
The DFE point of \eqref{eq:1} is globally asymptotically stable if $\mathcal{R}_0 < 1$.
\end{theorem}
\begin{proof}
In order show the GAS of the DEE point $E_0$ of \eqref{eq:1}, we consider \eqref{eq:1new} in place of \eqref{eq:1}. Note that the DFE point is now transformed to $\widehat{E}_0 = (I_0,\,N_0) = \left(0,\,\frac{\Lambda}{\mu}\right)$. Consider a Lyapunov function candidate defined by
\begin{equation*}
V(I, N) =l_1 I + \dfrac{1}{2}l_2(N - N_0)^2.
\end{equation*}
The time derivative of $V$ along with the solutions of \eqref{eq:1new} satisfies
\begin{equation*}
\begin{split}
\dfrac{\diff V}{\diff t} &= \dfrac{\diff V}{\diff I}\dfrac{\diff I}{\diff t} + \dfrac{\diff V}{\diff N}\dfrac{\diff N}{\diff t}\\
&= l_1\left[\dfrac{\beta(N - I)I}{h(N)} - (\mu+\delta+\gamma) I\right] + l_2(N - N_0)(\Lambda - \mu N - \delta I)\\
&= \left[l_1 \dfrac{\beta N}{h(N)} - l_1(\mu+\delta+\gamma) + l_2\delta N_0\right]I - l_2\mu(N - N_0)^2 -l_1\dfrac{\beta I^2}{h(N)} - l_2\delta NI\\
&\leq \left[l_1 \dfrac{\beta N}{h(N)} - l_1(\mu+\delta+\gamma) + l_2\delta N_0\right]I - l_2\mu(N - N_0)^2.
\end{split}
\end{equation*}
Since the function $\frac{N}{h(N)}$ is increasing for $N > 0$, we have the estimate on $\Omega$
\begin{equation*}
\dfrac{N}{h(N)} \leq \dfrac{\Lambda/\mu}{h(\Lambda/\mu)},
\end{equation*}
which implies that
\begin{equation}\label{eq:6a}
\dfrac{\diff V}{\diff t} \leq \left[l_1 \dfrac{\beta \Lambda/\mu}{h(\Lambda/\mu)} - l_1(\mu+\delta+\gamma) + l_2\delta \dfrac{\Lambda}{\mu}\right]I - \Lambda(N - N_0)^2 := L_1I - \mu(N - N_0)^2,
\end{equation}
where
\begin{equation*}
L_1 = l_1 \dfrac{\beta \Lambda/\mu}{h(\Lambda/\mu)} - l_1(\mu+\delta+\gamma) + l_2\delta \dfrac{\Lambda}{\mu}.
\end{equation*}
Note that $L_1$ can be rewritten in the form
\begin{equation*}
L_1 = l_1(\mu + \delta + \gamma)(\mathcal{R}_0 - 1) + l_2\delta\dfrac{\Lambda}{\mu}.
\end{equation*}
Since $\mathcal{R}_0<1$, one can always choose $l_1$ and $l_2$ such that
\begin{equation*}
0 < l_2 < l_1(\mu + \delta + \gamma)(1 - \mathcal{R}_0)\dfrac{\mu}{\delta\Lambda}.
\end{equation*}
For such a choice of $l_1$ and $l_2$, we have $L_1 < 0$. From \eqref{eq:6a}, we conclude that $\frac{\diff V}{\diff t} \leq 0$ for all $(I, N) \in \Omega$ with equality if and only if $(I,\,N) = \widehat{E}_0$. Using the Lyapunov stability theorem \cite{Khalil, Stuart}, the GAS of $\widehat{E}_0$ is established. The proof is complete.
\end{proof}
In the following theorem, we apply the results developed in \cite{Castillo-Chavez} to show the GAS of the DFE point.
\begin{theorem}[GAS analysis of the DFE point via the characteristic of $\mathcal{R}_0$]\label{Theorem3}
The DFE point of \eqref{eq:1} is globally asymptotically stable if $\mathcal{R}_0 < 1$.
\end{theorem}
\begin{proof}
First, we rewrite \eqref{eq:1new} in the form
\begin{equation*}
\begin{split}
&\dfrac{\diff X}{\diff t} = F(X, Z),\\
&\dfrac{\diff Z}{\diff t} = G(X, Z),
\end{split}
\end{equation*}
where $X = N$, $Z = I$ and
\begin{equation*}
\begin{split}
F(X, Z) &= \Lambda - \mu X - \delta Z,\\
G(X, Z) &= \dfrac{\beta(X - Z)Z}{h(X)} - (\mu+\delta+\gamma) Z.
\end{split}
\end{equation*}
Then, the dynamical system $\frac{\diff X}{\diff t} = F(X, 0)$ is reduced to
\begin{equation*}
\dfrac{\diff X}{\diff t} = \Lambda - \mu X.
\end{equation*}
This system has a unique equilibrium point $X^* = \frac{\Lambda}{\mu}$ and it is easily verified that it is globally asymptotically stable. Hence, the condition $(H1)$ in \cite{Castillo-Chavez} is satisfied.

On the other hand, we represent $G(X, Z)$ in the form
\begin{equation*}
G(X, Z) = AZ - \widehat{G}(X, Z),
\end{equation*}
where
\begin{equation*}
\begin{split}
&A := D_{Z}G(X^*, 0) = \dfrac{\beta X^*}{h(X^*)} - (\mu + \delta + \gamma),\\
&\widehat{G}(X, Z) = \dfrac{\beta Z^2}{h(X)} + \beta Z\left(\dfrac{X^*}{h(X^*)} - \dfrac{X}{h(X)}\right).
\end{split} 
\end{equation*}
If $\mathcal{R}_0 < 1$, then on the set $\Omega$ we have
\begin{equation*}
\begin{split}
&A = \dfrac{\beta X^*}{h(X^*)} - (\mu + \delta + \gamma) = \dfrac{\beta {(\Lambda}/{\mu})}{h\left({\Lambda}/{\mu}\right)} - (\mu + \delta + \gamma) < 0,\\
&\widehat{G}(X, Z) \geq 0.
\end{split}
\end{equation*}
which implies that the condition (H2) in \cite{Castillo-Chavez} holds.

Applying the global-stability result for the DFE in \cite{Castillo-Chavez} yields the GAS of the DFE point of \eqref{eq:1new}. The proof is complete.
\end{proof}
\section{Construction of Dynamically Consistent NSFD Schemes}\label{Sec3}
In this section, we construct first- and second-order NSFD schemes for the SIS model under consideration. 
\subsection{Construction of First-Order NSFD Schemes}
This section is devoted to the construction of first-order NSFD schemes. To this end, we first consider \eqref{eq:1} on a finite interval $[0,T]$ and partition this interval using a uniform grid:
\begin{equation*}
0 = t_0 < t_1 < \ldots < T_N = T,
\end{equation*}
where $t_k = k\Delta t (0 = 1, 2, \ldots, N)$ with $\Delta t = \dfrac{T}{N}$ being the step size.

First, we apply Mickens' methodology \cite{Mickens1, Mickens2, Mickens3, Mickens4, Mickens5} to construct a family of first-order NSFD schemes. Following this approach, the first derivatives are discretized by 
\begin{equation}\label{eq:NSFD1}
\begin{split}
&\dfrac{\diff S}{\diff t}\bigg|_{t = t_k} \approx \dfrac{S_{k + 1} - S_k}{\phi(\Delta t)},\\
&\dfrac{\diff I}{\diff t}\bigg|_{t = t_k} \approx \dfrac{I_{k + 1} - I_k}{\phi(\Delta t)},
\end{split}
\end{equation}
where $S_k$ and $I_k$ ($k = 1, 2, \ldots, N$) are the intended approximations for $S(t_k)$ and $I(t_k)$, respectively; $\phi(\Delta t) = \Delta t + \mathcal{O}(\Delta t^2)$ as $\Delta t \to 0$ is called a denominator function. Meanwhile, the right-hand side at $t = t_k$ is approximated by
\begin{equation}\label{eq:NSFD2}
\begin{split}
&\Lambda - \dfrac{{\beta} S(t_k) I(t_k)}{h(N(t_k))}- \mu S(t_k) + \gamma I(t_k) \approx \Lambda - \dfrac{{\beta} S_{k + 1} I_k}{h(N_k)}- \mu S_{k + 1} + \gamma I_k,\\
&\dfrac{{\beta} S(t_k) I(t_k)}{h(N(t_k))} - (\mu+\delta+\gamma) I(t_k) \approx \dfrac{{\beta} S_{k+1} I_k}{h(N_k)} - (\mu+\delta) I_{k + 1} - \gamma I_k.
\end{split}
\end{equation}
Combining \eqref{eq:NSFD1} and \eqref{eq:NSFD2} leads to the family of NSFD schemes
\begin{equation}\label{eq:NSFD1st}
\begin{split}
&\dfrac{S_{k + 1} - S_k}{\phi(\Delta t)} = \Lambda - \dfrac{{\beta} S_{k + 1} I_k}{h(N_k)}- \mu S_{k + 1} + \gamma I_k,\\
&\dfrac{I_{k + 1} - I_k}{\phi(\Delta t)} = \dfrac{{\beta} S_{k+1} I_k}{h(N_k)} - (\mu+\delta) I_{k + 1} - \gamma I_k.
\end{split}
\end{equation}
Next, we analyze the positivity of the approximations generated by \eqref{eq:NSFD1st} and determine the equilibrium points of the NSFD schemes. To this end, we impose the following condition on the denominator function:
\begin{equation}\label{eq:NSFD4}
\phi(h) < \phi_P := (\gamma)^{-1}, \quad \forall h > 0.
\end{equation}
\begin{theorem}[The positivity of the approximations]\label{Theorem1NSFD}
Let $S_0, I_0 \geq 0$ be arbitrary initial data for \eqref{eq:1}, and assume that the condition \eqref{eq:NSFD4} holds. Then, for any step size, the approximations generated by \eqref{eq:NSFD1st} satisfy $S_k, I_k \geq 0$ for all $k \geq 0$. In other words, NSFD schemes of the form \eqref{eq:NSFD1st} are dynamically consistent with respect to the positivity of the solutions of \eqref{eq:1}.
\end{theorem}
\begin{proof}
The NSFD schemes of the form \eqref{eq:NSFD1st} can be rewritten in the form
\begin{equation}\label{eq:5new}
\begin{split}
&S_{k + 1} = \dfrac{S_k + \phi \Lambda + \phi\gamma I_k}{1 + \phi\dfrac{{\beta}I_k}{h(N_k)} + \phi\mu},\\
&I_{k + 1} = \dfrac{\left(1 - \phi\gamma\right)I_k + \phi\dfrac{{\beta} S_{k+1} I_k}{h(N_k)}}{1 + \phi(\mu + \delta)}.
\end{split}
\end{equation}
Since $\phi(h)$ satisfies \eqref{eq:NSFD4}, $1 - \phi\gamma > 0$. Combining this with \eqref{eq:5new} implies that $S_{k+1}, I_{k + 1} \geq 0$ whenever $S_k, I_k \geq 0$. Thus, by mathematical induction, we conclude that $S_k, I_k \geq 0$ for $k > 0$ if $S_0, I_0 \geq 0$. This is the desired conclusion. The proof is complete.
\end{proof}
Throughout this section, we always assume that \eqref{eq:NSFD4} holds. From \eqref{eq:5new}, by considering the system $S_{k + 1} = S_k$ and $I_{k + 1} = I_k$, we conclude that the sets of equilibrium points of \eqref{eq:1} and \eqref{eq:NSFD1st} are identical. Thus, the NSFD model \eqref{eq:NSFD1st} always possesses a DFE point for all values of the model parameters, whereas a DEE point $E_*$ exists if and only if $\mathcal{R}_0 > 1$, where $E_0$ and $E_*$ are defined as in Section \ref{Sec2}. Note that a direct calculation using the method developed by Allen and van den Driessche in \cite{Allen1} shows that the basic reproduction number of \eqref{eq:NSFD1st} coincides with that of \eqref{eq:1}.

To analyze the local asymptotic stability of the equilibrium points of \eqref{eq:NSFD1st}, we apply the linearization method \cite{Allen, Stuart} with the help of the Schur--Cohn criterion \cite{Allen}. Following the approaches introduced in \cite{DangHoang2018, Dimitrov6, Wood1}, we derive a local asymptotic stability (LAS) threshold $\phi_{\mathrm{LAS}}$ for \eqref{eq:NSFD1st} that depends only on the parameter values of \eqref{eq:1}. More precisely, \eqref{eq:NSFD1st} preserves the LAS of the corresponding equilibrium points of \eqref{eq:1} whenever
\begin{equation}\label{eq:6new}
\phi(\Delta t) < \phi_{LAS}, \quad \forall \Delta t > 0.
\end{equation}
Combining \eqref{eq:NSFD4} and \eqref{eq:6new} leads to the condition for the denominator function
\begin{equation}\label{eq:6b}
\phi(\Delta t) < \phi_{DC} := \min\{\phi_P,\,\phi_{LAS}\}, \quad \forall \Delta t > 0.
\end{equation}

Note that there are many functions satisfying \eqref{eq:6b}, for instance:
\begin{equation*}
\phi(h) = \dfrac{1 - e^{-ch}}{c}, \quad c > \dfrac{1}{\phi_{DC}^{-1}}.
\end{equation*}
Establishing the GAS of equilibrium points is challenging for NSFD schemes in general and for schemes of the form \eqref{eq:NSFD1st} in particular. Nevertheless, the numerical simulations presented in the next section suggest that the NSFD schemes of the form \eqref{eq:NSFD1st} preserve the GAS whenever the condition for the LAS preservation is satisfied.

Using the error analysis presented in \cite{Cresson}, we conclude that the NSFD schemes of the form \eqref{eq:NSFD1st} are only first-order accurate and therefore converges with order one. The numerical simulations in the next section will illustrate this assertion.
\subsection{Construction of Second-Order NSFD Schemes}
In this subsection, we use the approach in \cite{HoangMatthias2} to construct a family of second-order NSFD schemes for \eqref{eq:1}. Following this approach, we extend the first-order NSFD schemes of the form \eqref{eq:NSFD1st} as follows
\begin{equation}\label{eq:NSFD2nd}
\begin{split}
&\dfrac{S_{k + 1} - S_k}{\Phi_1(\Delta t, S_k, I_k)} = \Lambda - \dfrac{{\beta} S_{k + 1} I_k}{h(N_k)}- \mu S_{k + 1} + \gamma I_k + \tau_1 S_k - \tau_1 S_{k + 1},\\
&\dfrac{I_{k + 1} - I_k}{\Phi_2(\Delta t, S_k, I_k)} = \dfrac{{\beta} S_{k} I_k}{h(N_k)} - (\mu+\delta + \gamma) I_{k + 1} + \tau_2 I_k - \tau_2 I_{k + 1},
\end{split}
\end{equation}
where $0 < \Phi_i(\Delta, S, I) = \Delta t + \mathcal{O}{(\Delta t^2)}$ as $\Delta t \to 0$ for all $S, I \geq 0$.
The main difference between \eqref{eq:NSFD2nd} and \eqref{eq:NSFD1st} is that the former employs two distinct denominator functions in the discrete derivative approximations; moreover, the approximation of the right-hand side incorporates the additional weighted discrete terms $\tau_1(S_k-S_{k+1})$ and $\tau_2(I_k-I_{k+1})$. Consequently, \eqref{eq:NSFD2nd} is more flexible and can achieve second-order convergence through appropriate choices of the denominator functions, while its dynamic consistency is ensured by the weights $\tau_i$ $(i=1,2)$. The details are presented below.
\begin{theorem}[The positivity of the approximations]\label{Theorem2NSFD}
Let $S_0, I_0 \geq 0$ be arbitrary initial data for \eqref{eq:1} and $\tau_1$ and $\tau_2$ be real numbers satisfying
\begin{equation}\label{eq:7}
\tau_1 \geq \tau_1^P := 0, \quad \tau_2 \geq \tau_2^P :=0.
\end{equation}
Then, for any step size, the approximations generated by \eqref{eq:NSFD2nd} satisfy $S_k, I_k \geq 0$ for all $k \geq 0$. In other words, the NSFD schemes of the form \eqref{eq:NSFD2nd} are dynamically consistent with respect to the positivity of the solutions of \eqref{eq:1}.
\end{theorem}
\begin{proof}
First, we transform \eqref{eq:NSFD2nd} into the form
\begin{equation}\label{eq:8}
\begin{split}
&S_{k + 1} = \dfrac{(1 + \Phi_1\tau_1)S_k + \Phi_1\Lambda + \Phi_1\gamma I_k}{1 + \Phi_1\dfrac{{\beta}I_k}{h(N_k)} + \Phi_1\mu + \Phi_1\tau_1},\\
&I_{k + 1} = \dfrac{\left(1 + \Phi_2\tau_2\right)I_k + \Phi_2\dfrac{{\beta} S_{k} I_k}{h(N_k)}}{1 + \Phi_2(\mu + \gamma+ \delta) + \Phi_2\tau_2}.
\end{split}
\end{equation}
Thus, it follows from \eqref{eq:7} and \eqref{eq:8} that $S_{k+1}, I_{k + 1} \geq 0$ whenever $S_k, I_k \geq 0$. By mathematical induction, we conclude that $S_k, I_k \geq 0$ for $k > 0$ if $S_0, I_0 \geq 0$. This is the desired conclusion. The proof is complete.
\end{proof}
Throughout this section, we always assume that \eqref{eq:7} is satisfied. From \eqref{eq:8}, it is easily verified that the sets of equilibrium points of \eqref{eq:1} and \eqref{eq:NSFD2nd} are identical.

Using the approach and calculations presented in \cite{HoangMatthias2}, we derive stability thresholds $\tau_i^{LAS} > 0$ for the parameters $\tau_i$ in \eqref{eq:NSFD2nd}. More precisely, \eqref{eq:NSFD2nd} preserves the LAS of the corresponding equilibrium points of \eqref{eq:1} whenever
\begin{equation}\label{eq:9}
\tau_1 \geq \tau_1^{LAS},\quad \quad \tau_2  \geq \tau^{LAS}_2.
\end{equation}
From \eqref{eq:7} and \eqref{eq:9}, we obtain a condition imposed on $\tau_i$, which ensures the dynamic consistency of \eqref{eq:NSFD2nd} and is given by
\begin{equation*}
\tau_1 \geq \max\left\{\tau_1^{P},\,\,\,\tau_1^{LAS}\right\},\quad \tau_2 \geq \max\left\{\tau_2^{P},\,\,\,\tau_2^{LAS}\right\}.
\end{equation*}
A condition imposed on the denominator functions to ensure that the NSFD schemes of the form \eqref{eq:NSFD2nd} are second-order accurate is formulated below.
\begin{theorem}[Second-order NSFD schemes]\label{Theorem3}
Let $\Phi_1(\Delta t, S, I)$ and $\Phi_2(\Delta t, S, I)$ be functions satisfying the following condition
\begin{equation}\label{eq:10}
\begin{split}
&\dfrac{\partial^2 \Phi_1}{\partial \Delta t^2}(0, S, I)= D_1(S, I) := 2\left(\dfrac{{\beta}I}{h(S + I)} + \mu + \tau_1\right)+\dfrac{\partial F(S, I)}{\partial S}+\dfrac{\partial F(S, I)}{\partial I} \dfrac{G(S, I)}{F(S, I)},\\
&\dfrac{\partial^2 \Phi_2}{\partial \Delta t^2}(0, S, I)= D_2(S, I) := 2\left(\delta + \gamma + \mu + \tau_2\right) + \dfrac{\partial G(S, I)}{\partial S}\dfrac{F(S, I)}{G(S, I)} + \dfrac{\partial G(S, I)}{\partial I},
\end{split}
\end{equation}
for all $(S, I) \in \mathbb{R}_{+}^2$ such that $F(S, I), G(S, I) \neq 0$, where $\left(F(S, I),\, G(S, I)\right)^{\top}$ is the right-hand side function of the model \eqref{eq:1}. Then, the truncation error of the NSFD schemes of the form \eqref{eq:NSFD2nd} is $\mathcal{O}\left(\Delta t^3\right)$, i.e. the schemes are consistent of order $2$.
\end{theorem}
\begin{proof}
First, using Taylor expansion for the solution components $S(t)$ and $I(t)$ yields
\begin{equation}\label{eq:11}
\begin{split}
S\left(t_{k+1}\right) = S\left(t_k + \Delta t\right) &= S\left(t_k\right)+\Delta t S^{\prime}\left(t_k\right)+\dfrac{\Delta t^2}{2} S^{\prime \prime}\left(t_k\right)+\mathcal{O}\left(\Delta t^3\right)\\
& =S\left(t_k\right)+\Delta t F\left(S\left(t_k\right), I\left(t_k\right)\right)+\dfrac{\Delta t^2}{2} \dfrac{\partial F\left(S\left(t_k\right), I\left(t_k\right)\right)}{\partial t}+\mathcal{O}\left(\Delta t^3\right), \\
I\left(t_{k+1}\right) = I\left(t_k + \Delta t\right)& =I\left(t_k\right)+\Delta t I^{\prime}\left(t_k\right)+\dfrac{\Delta t^2}{2} I^{\prime \prime}\left(t_k\right)+\mathcal{O}\left(\Delta t^3\right) \\
& =I\left(t_k\right)+\Delta t G\left(S\left(t_k\right), I\left(t_k\right)\right)+\dfrac{\Delta t^2}{2} \dfrac{\partial G\left(S\left(t_k\right), I\left(t_k\right)\right)}{\partial t}+\mathcal{O}\left(\Delta t^3\right).
\end{split}
\end{equation}
Let us denote by $\left(F_D\left(\Delta t, S_k, I_k\right), G_D\left(\Delta t, S_k, I_k\right)\right)^{\top}$ the right-side function of \eqref{eq:8}. It follows from \eqref{eq:8} that
\begin{equation}\label{eq:12}
\begin{split}
&F_D(0, S, I) = S,\\
&\dfrac{\partial F_D(0, S, I)}{\partial \Delta t} = F(S, I),\\
&\dfrac{\partial^2 F_D(0, S, I)}{\partial \Delta t^2} = F(S, I)\left[\dfrac{\partial^2 \Phi_1(0, S, I)}{\partial \Delta t^2} - 2\left(\dfrac{{\beta}I}{h(S + I)} + \mu + \tau_1\right)\right],\\
&G_D(0, S, I) = I, \\
&\dfrac{\partial G_D(0, S, I)}{\partial \Delta t}=G(S, I),\\
&\dfrac{\partial^2 G_D(0, S, I)}{\partial \Delta t^2}=G(S, I)\left[\dfrac{\partial^2 \Phi_2(0, S, I)}{\partial \Delta t^2} - 2\left(\delta + \gamma + \mu +  \tau_2\right)\right].
\end{split}
\end{equation}
Combining \eqref{eq:12} with Taylor expansion theorem gives
\begin{equation}\label{eq:13}
\begin{split}
S_{k+1}= & F_D\left(\Delta t, S_k, I_k\right) = F_D(0, S_k, I_k) + \Delta t \dfrac{\partial F_D(0, S_k, I_k)}{\partial \Delta t} +\dfrac{\Delta t^2}{2} \dfrac{\partial^2 F_D(0, S_k, I_k)}{\partial \Delta t^2} + \mathcal{O}\left(\Delta t^3\right),\\
= &S_k + \Delta t F\left(S_k, I_k\right) + \dfrac{\Delta t^2}{2} F\left(S_k, I_k\right){\left[\dfrac{\partial^2 \Phi_1\left(0, S_k, I_k\right)}{\partial \Delta t^2} - 2\left(\dfrac{{\beta}I}{h(S_k + I_k)} + \mu + \tau_1\right)\right]+\mathcal{O}\left(\Delta t^3\right), }\\
I_{k+1}= & G_D\left(\Delta t, S_k, I_k\right) = G_D(0, S_k, I_k) + \Delta t \dfrac{\partial G_D(0, S_k, I_k)}{\partial \Delta t} +\dfrac{\Delta t^2}{2} \dfrac{\partial^2 G_D(0, S_k, I_k)}{\partial \Delta t^2} + \mathcal{O}\left(\Delta t^3\right)\\
= &I_k + \Delta t G\left(S_k, I_k\right)+\dfrac{\Delta t^2}{2} G\left(S_k, I_k\right)\left[\dfrac{\partial^2 \Phi_2\left(0, S_k, I_k\right)}{\partial \Delta t^2} - 2\left(\delta + \gamma+ \mu + \tau_2\right)\right]+\mathcal{O}\left(\Delta t^3\right).
\end{split}
\end{equation}
Thus, we deduce from \eqref{eq:11} and \eqref{eq:13} that
\begin{equation*}
S_{k+1} - S\left(t_{k+1}\right) = \mathcal{O}\left(\Delta t^3\right), \quad I_{k+1} - I\left(t_{k+1}\right) = \mathcal{O}\left(\Delta t^3\right)
\end{equation*}
if \eqref{eq:10} holds. This is the desired conclusion and the proof is complete.
\end{proof}
\begin{remark}
A family of denominator functions satisfying \eqref{eq:10} is given by
\begin{equation}\label{eq:14}
\Phi(\Delta t, S, I) =
\begin{cases}
&\dfrac{e^{D(S, I)\Delta t} - 1}{D(S, I)} \quad \mbox{if} \quad D(S, I) \ne 0,\\
&\Delta t \quad \mbox{if} \quad D(S, I) = 0.
\end{cases}
\end{equation}
The functions of this form will be used in numerical examples reported in the next section.
\end{remark}
\section{Numerical Examples}\label{Sec4}
In this section, we conduct a set of numerical examples to support the theoretical findings. First, we provide an error analysis of the constructed first-order and second-order NSFD schemes. In the numerical examples reported below, we will use the first-order NSFD scheme of the form \eqref{eq:NSFD1st} with $\phi(\Delta t) = 1 - e^{-\Delta t}$ and the second-order NSFD scheme of the form \eqref{eq:NSFD2nd} with $\tau_1 = \tau_2 = 1$. 
\begin{example}[An Error Analysis of the NSFD schemes]
In this example, we compute the errors produced by the constructed first- and second-order NSFD schemes when solving \eqref{eq:1} with the following parameter set:
\begin{equation*}
\Lambda = 100, \quad \mu = 2.5 \times 10^{-4}, \quad \gamma =  0.7, \quad \delta = 10^{-5}, \quad \beta = 0.05, \quad  b = 0.05 
\end{equation*}
and initial data
\begin{equation*}
S(0) = 35 \times 10^4,\quad I(0) = 10^3.
\end{equation*}
Since the exact solution cannot be determined analytically, we apply the classical four-stage Runge--Kutta method \cite{Ascher, Stuart} with a step size of $\Delta t=10^{-6}$ to generate a reference solution on the interval $[0, 1]$. Then, the errors are computed as follows:
\begin{equation*}
\begin{split}
&\err_R(S) = \dfrac{\left|S_N - S(t_N)\right|}{|S(t_N)|},\\
&\err_R(I) = \dfrac{\left|I_N - I(t_N)\right|}{|I(t_N)|},\\
&\err_R(S, I) = \dfrac{\left|S_N - S(t_N)\right| + \left|I_N - I(t_N)\right|}{\left|S(t_N)\right| + \left|I(t_N)\right|},\\
&\err_F = \left|S_N - S(t_N)\right| + \left|I_N - I(t_N)\right|,
\end{split}
\end{equation*}
where $t_N = 1$ and $\left(S(t_N), I(t_N)\right)$ standards for the reference solution. Moreover, the rate of convergence (ROC) is estimated by \cite{Ascher}
\begin{equation*}
\ROC := \log_{\bigg(\dfrac{\Delta t_1}{\Delta t_2}\bigg)}\bigg(\dfrac{\err_F(\Delta t_1)}{\err_F(\Delta t_2)}\bigg).
\end{equation*}
The numerical results for the first-order and second-order NSFD schemes and the explicit Euler scheme \cite{Ascher, Stuart} are reported in Tables \ref{Table1}-\ref{Table3}. Clearly, the obtained error estimates and ROCs are consistent with the theoretical analysis presented in Section \ref{Sec3}.
\begin{table}[H]
\begin{center}
\caption{The errors and ROC of the first-order NSFD scheme}\label{Table1}
\begin{tabular}{ccccccccccc}
\hline
$\Delta t$&$\err_R(S)$&$\err_R(I)$&$\err_R(S, I)$&$\err_F$&ROC\\
\hline
$10^{-1}$ & 7.5884e-005 & 0.0205 &  1.5328e-004 &  53.8043 & \\
\hline
$5 \times 10^{-2}$ & 3.9278e-005 & 0.0106 & 7.9325e-005 &  27.8441 & 0.9504\\
\hline
$10^{-2}$ & 8.0825e-006 & 0.0022 & 1.6321e-005 & 5.7288 &  0.9824\\
\hline
$5 \times 10^{-3}$ & 4.0559e-006 & 0.0011 & 8.1898e-006 &  2.8747 & 0.9948\\
\hline 
$10^{-3}$ & 8.1353e-007 & 2.1949e-004 & 1.6427e-006 & 0.5766 & 0.9982\\
\hline
$5 \times 10^{-4}$ & 4.0691e-007 & 1.0979e-004 & 8.2164e-007 &  0.2884 & 0.9995 \\
\hline
$10^{-4}$ & 8.1407e-008 &  2.1963e-005 &  1.6438e-007 &  0.0577 & 0.9998\\
\hline
$5 \times 10^{-5}$ & 4.0764e-009 &  1.0982e-006 &  8.2250e-009 & 0.0029 & 0.9999\\
\hline
\end{tabular}
\end{center}
\end{table}
\end{example}
\begin{table}[H]
\begin{center}
\caption{The errors and ROC of the second-order NSFD scheme}\label{Table2}
\begin{tabular}{ccccccccccc}
\hline
$\Delta t$&$\err_R(S)$&$\err_R(I)$&$\err_R(S, I)$&$\err_F$&ROC\\
\hline
$10^{-1}$ &  5.3788e-006 &  0.0032 & 1.7335e-005 & 6.0849 &\\
\hline
$5 \times 10^{-2}$ & 1.3650e-006 & 7.9855e-004 & 4.3876e-006 & 1.5401 &  1.9822\\
\hline
$10^{-2}$ & 5.5084e-008 &  3.2056e-005 & 1.7642e-007 & 0.0619 & 1.9968 \\
\hline
$5 \times 10^{-3}$ & 1.3783e-008 & 8.0150e-006 & 4.4121e-008 & 0.0155 & 1.9995 \\
\hline
$10^{-3}$ & 5.5170e-010 & 3.2061e-007 & 1.7652e-009 & 6.1962e-004 &  1.9999\\
\hline
$5 \times 10^{-4}$ & 1.3793e-010 & 8.0153e-008 &  4.4132e-010 & 1.5491e-004 & 2.0000\\
\hline
$10^{-4}$ & 5.5138e-012 & 3.2061e-009 & 1.7649e-011 & 6.1950e-006 &    2.0001\\
\hline
$5 \times 10^{-5}$ & 1.4064e-012 & 8.0149e-010 & 4.4400e-012& 1.5585e-006& 1.9909\\
\hline
\end{tabular}
\end{center}
\end{table}
\begin{table}[H]
\begin{center}
\caption{The errors and ROC of the standard explicit Euler scheme}\label{Table3}
\begin{tabular}{ccccccccccc}
\hline
$\Delta t$&$\err_R(S)$&$\err_R(I)$&$\err_R(S, I)$&$\err_F$&ROC\\
\hline
$10^{-1}$ &   1.5038e-005 & 0.0040 &  2.9960e-005 & 10.5164 &\\
\hline
$5 \times 10^{-2}$ & 7.5955e-006 & 0.0020 & 1.5133e-005 & 5.3118 & 0.9854\\
\hline
$10^{-2}$ &  1.5316e-006 & 4.0239e-004 & 3.0515e-006 & 1.0711 &  0.9949\\
\hline
$5 \times 10^{-3}$ & 7.6660e-007 &  2.0140e-004 & 1.5273e-006 & 0.5361 & 0.9985 \\
\hline
$10^{-3}$ &  1.5345e-007 & 4.0314e-005 &  3.0572e-007 &  0.1073 & 0.9995\\
\hline
$5 \times 10^{-4}$ & 7.6731e-008 & 2.0159e-005 & 1.5288e-007 & 0.0537 & 0.9999\\
\hline
$10^{-4}$ & 1.5348e-008 &  4.0321e-006 & 3.0578e-008 &  0.0107 & 0.9999\\
\hline
$5 \times 10^{-5}$ & 7.6738e-009 & 2.0161e-006 & 1.5289e-008 & 0.0054 &  1.0000\\
\hline
\end{tabular}
\end{center}
\end{table}
In the following example, we will compare the long-term numerical dynamics of the NSFD schemes and the standard explicit Euler scheme.
\begin{example}[A comparison of the dynamics of the NSFD and standard Euler schemes]
In this example, we provide a comparison of the dynamics of the NSFD and standard Euler schemes. For this purpose, we consider the SIS model \eqref{eq:1} with the following parameters
\begin{equation*}
\Lambda = 100, \quad \mu  = 2.5 \times 10^{-4}, \quad \gamma  = 0.95, \quad \delta  = 10^{-5}, \quad b      = 0.05, \quad \beta   = 0.18
\end{equation*}
and the initial data
\begin{equation*}
S(0) = 35 \times 10^4,\quad I(0) = 10^3.
\end{equation*}
Note that for this parameter set $\mathcal{R}_0 = 3.7507  > 1$. Thus, the DEE point,\\ $E_* = \left(1.0362 \times 10^5,\,2.8498 \times 10^5\right)$ is globally asymptotically stable.
The approximate solutions generated by the numerical schemes for several step sizes are depicted in Figures \ref{Fig:1}--\ref{Fig:4}. It is evident that the standard Euler scheme produces unstable approximations that oscillate around the equilibrium positions. Thus, the dynamical properties of the continuous--time model are not preserved for arbitrary step sizes. In contrast, both the first- and second-order NSFD schemes accurately reproduce the dynamics of the continuous--time model for all the selected step sizes, including relatively large ones. These numerical results are consistent with the theoretical analysis presented in Section \ref{Sec3} and demonstrate the advantages of the NSFD schemes.
\begin{figure}[H]
\subfloat[$S$-component]{%
\includegraphics[height=9cm,width=15cm]{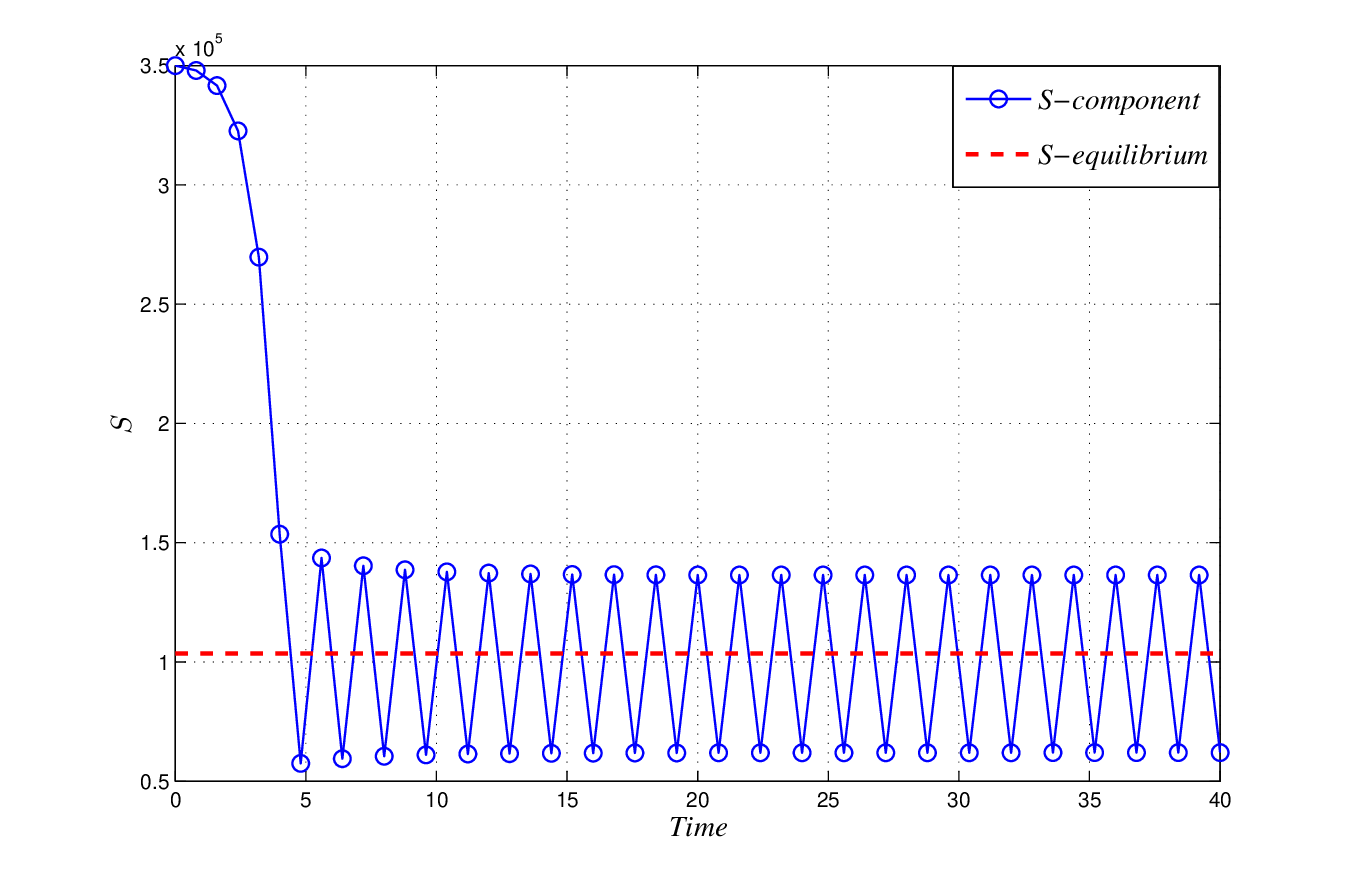}
\label{Figure:1a}
}\hfill
\subfloat[$I$-component]{%
\includegraphics[height=9cm,width=15cm]{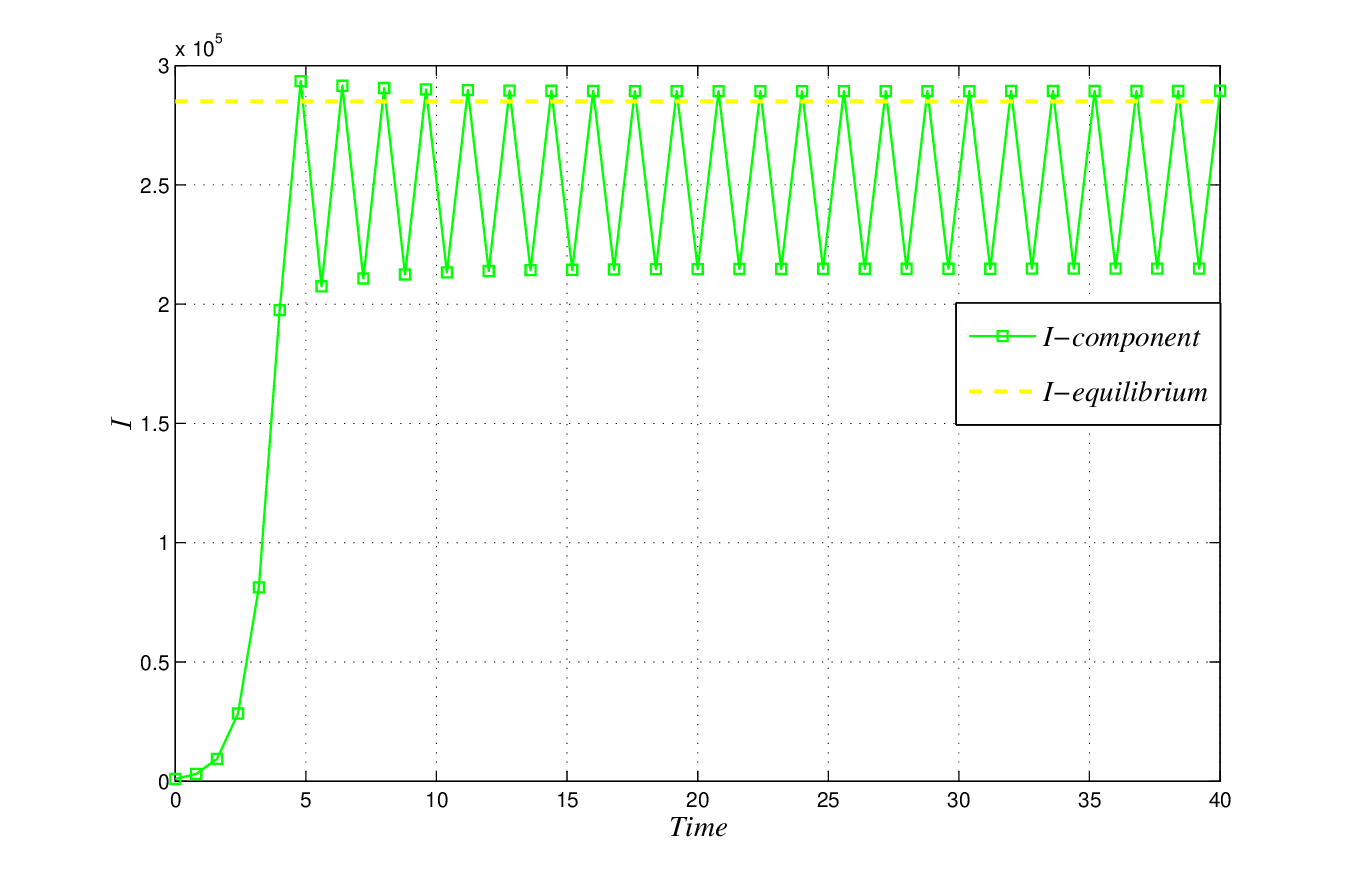}
\label{Figure:1b}
}\hfill
\caption{The numerical approximations generated by the standard explicit Euler scheme with $\Delta t = 0.8$.}\label{Fig:1}
\end{figure}
\begin{figure}[H]
\subfloat[$S$-component]{%
\includegraphics[height=10cm,width=15cm]{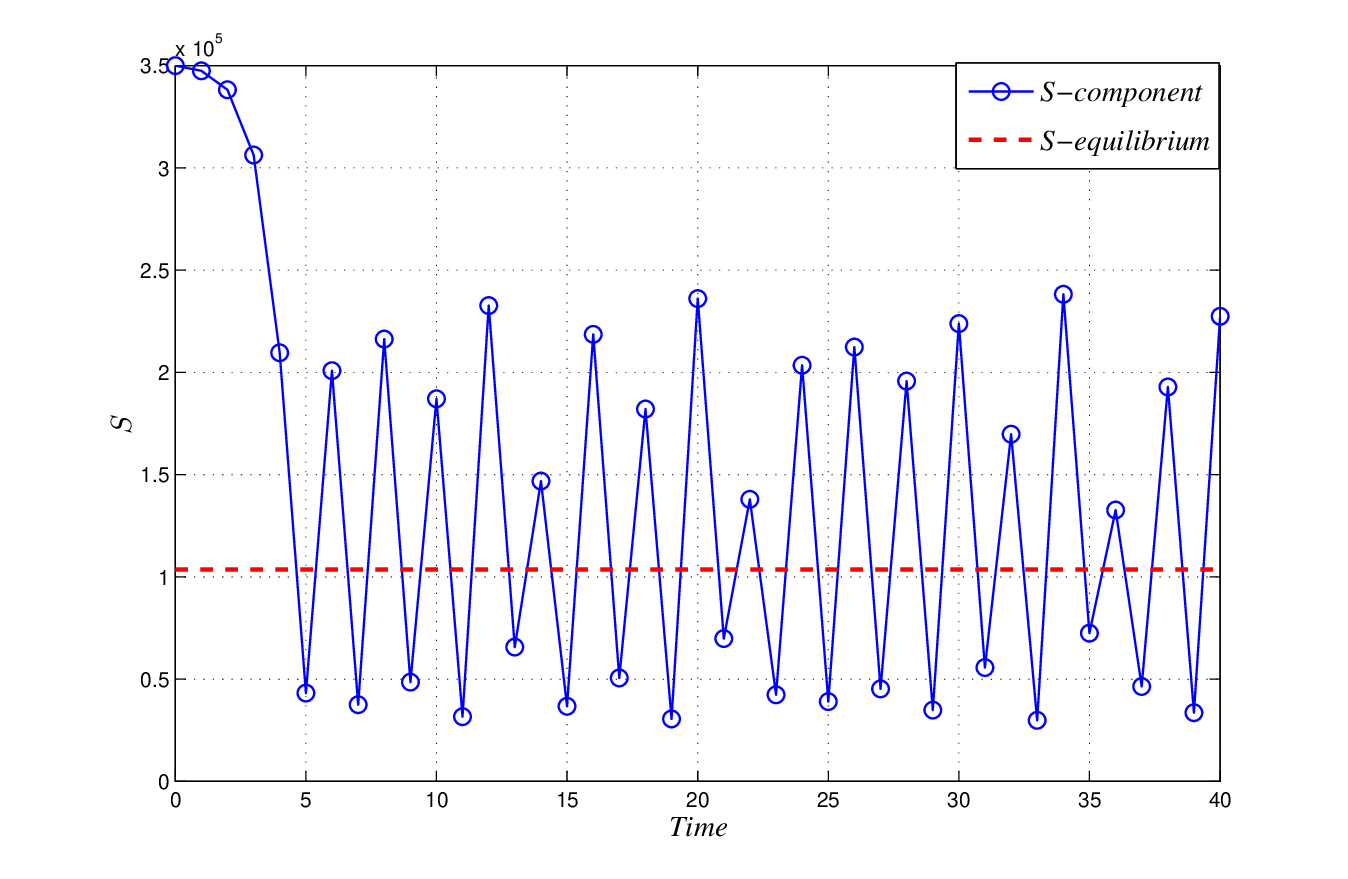}
\label{Figure:2a}
}\hfill
\subfloat[$I$-component]{%
\includegraphics[height=10cm,width=15cm]{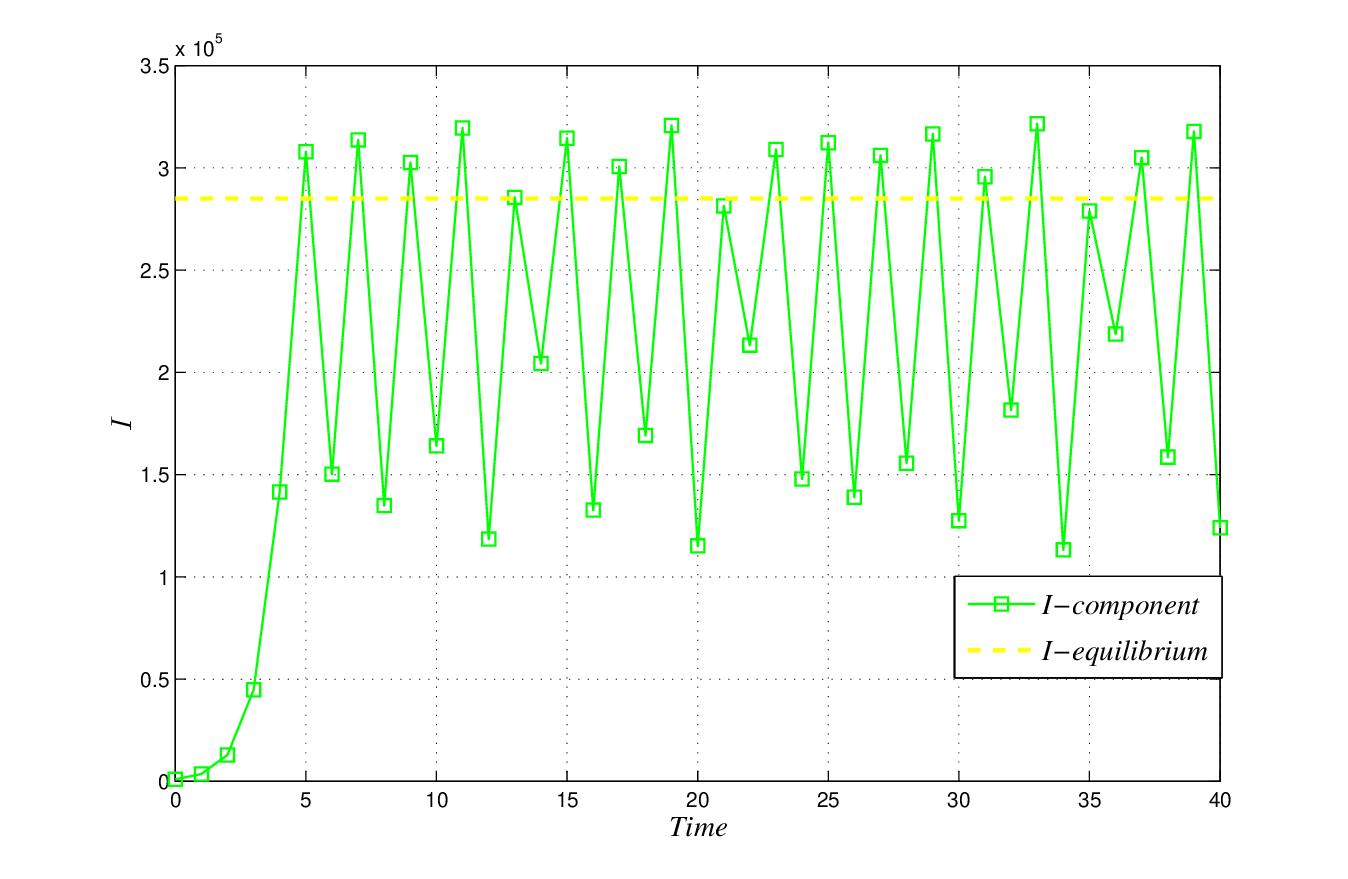}
\label{Figure:2b}
}\hfill
\caption{The numerical approximations generated by the standard explicit Euler scheme with $\Delta t = 1.0$.}\label{Fig:2}
\end{figure}
\begin{figure}[H]
\subfloat[$S$-component]{%
\includegraphics[height=10cm,width=15cm]{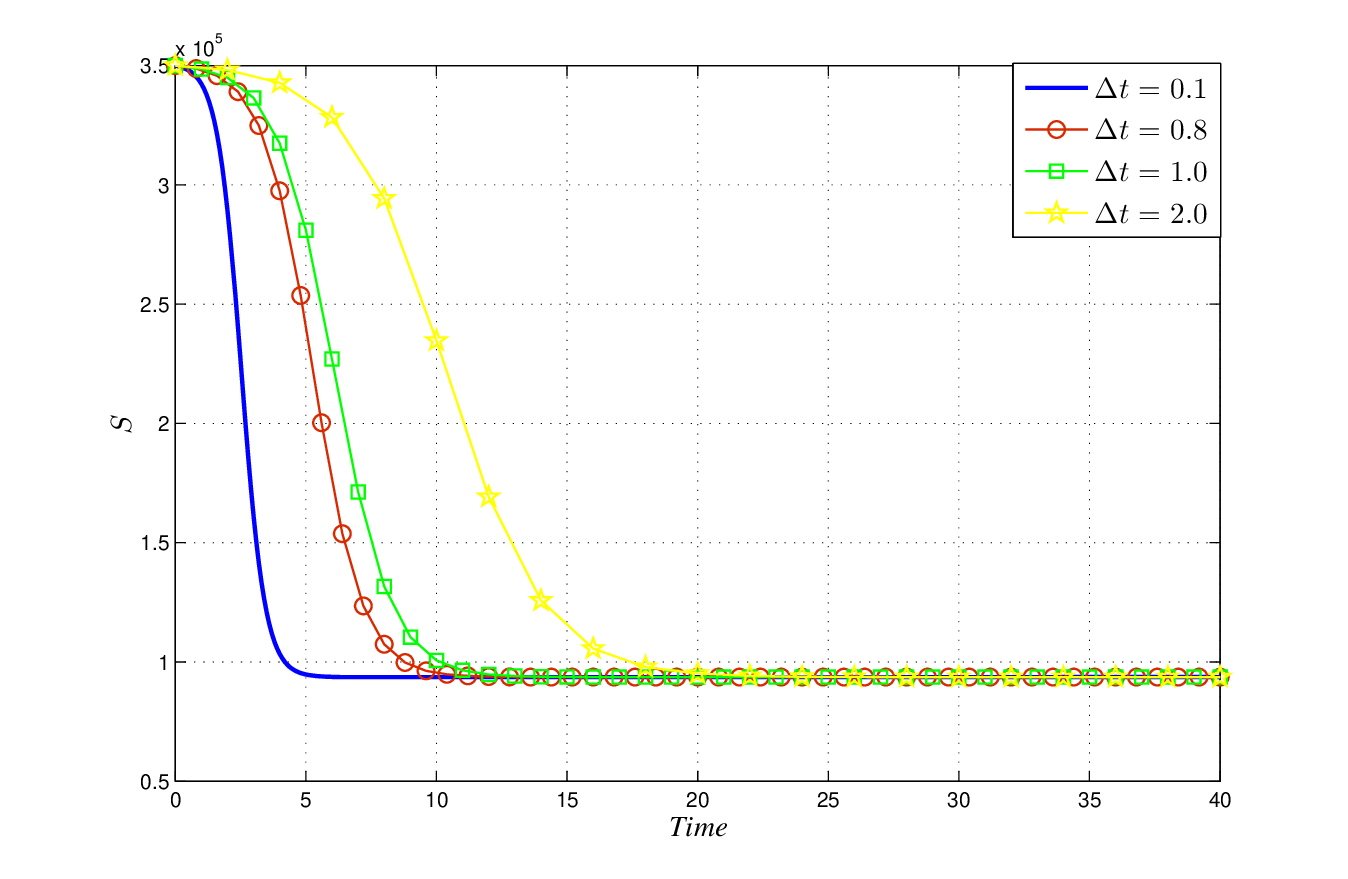}
\label{Figure:3a}
}\hfill
\subfloat[$I$-component]{%
\includegraphics[height=10cm,width=15cm]{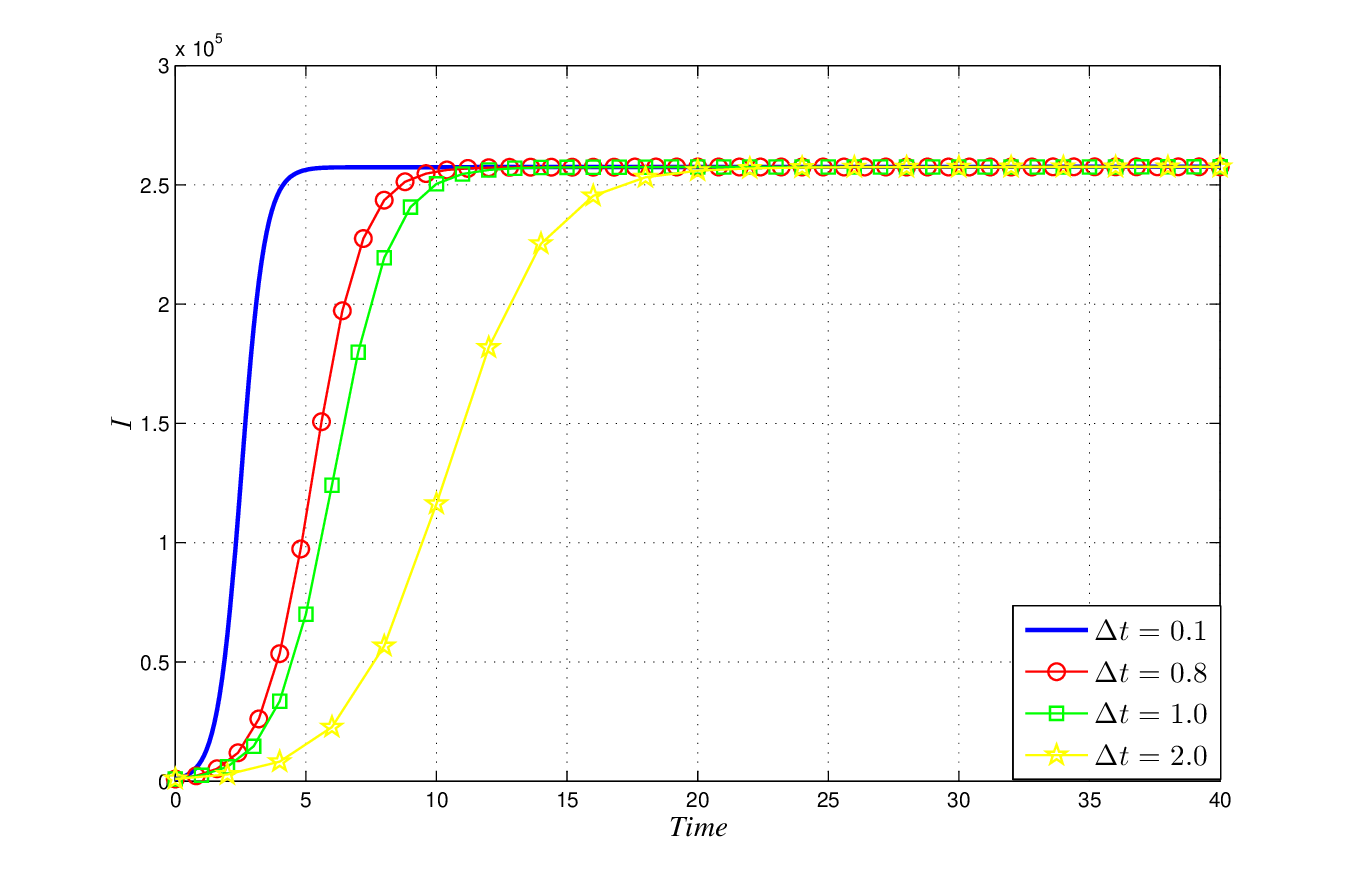}
\label{Figure:3b}
}\hfill
\caption{The numerical approximations generated by the first-order NSFD scheme with some different step sizes.}\label{Fig:3}
\end{figure}
\begin{figure}[H]
\subfloat[$S$-component]{%
\includegraphics[height=10cm,width=15cm]{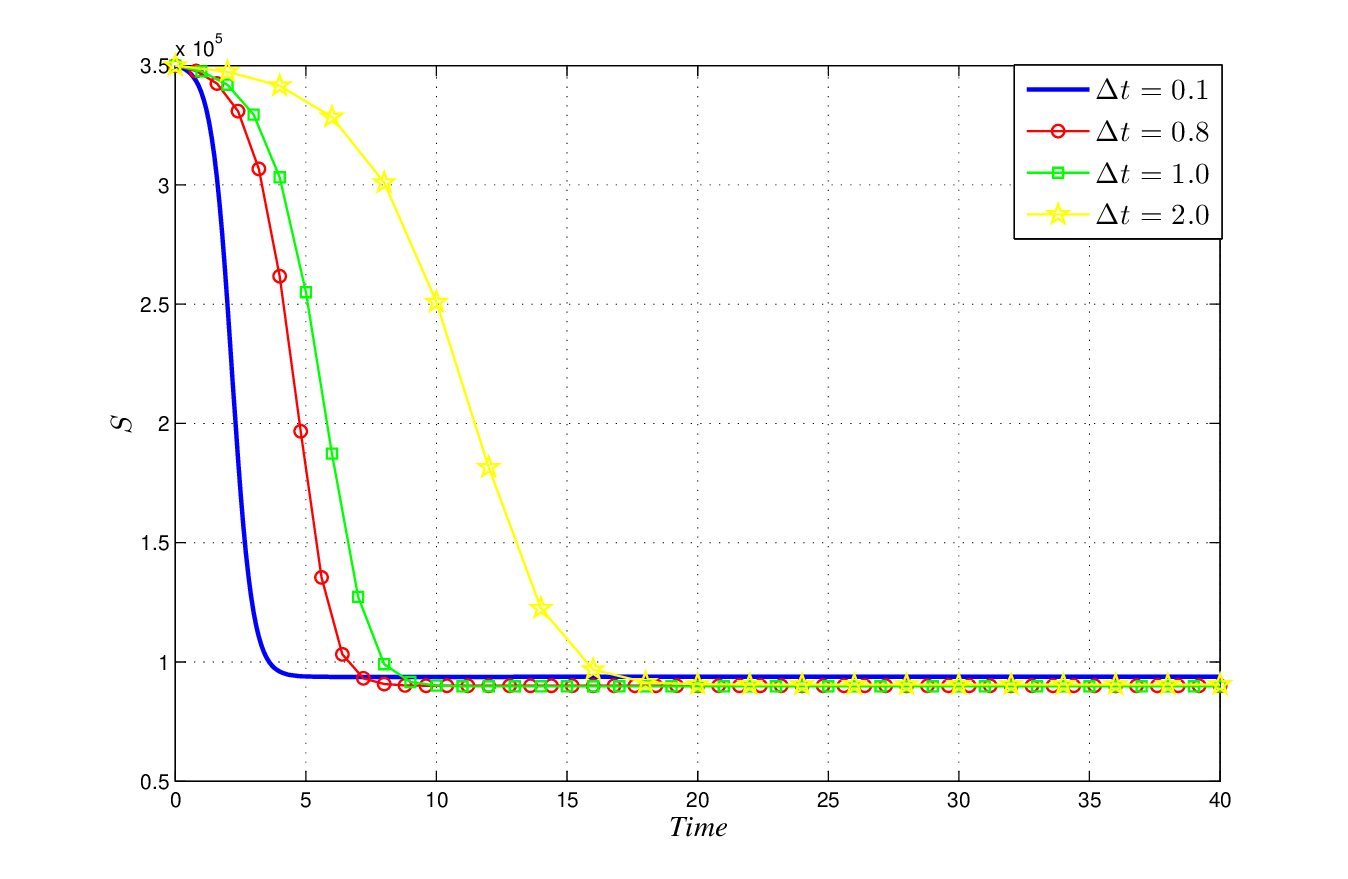}
\label{Figure:4a}
}\hfill
\subfloat[$I$-component]{%
\includegraphics[height=10cm,width=15cm]{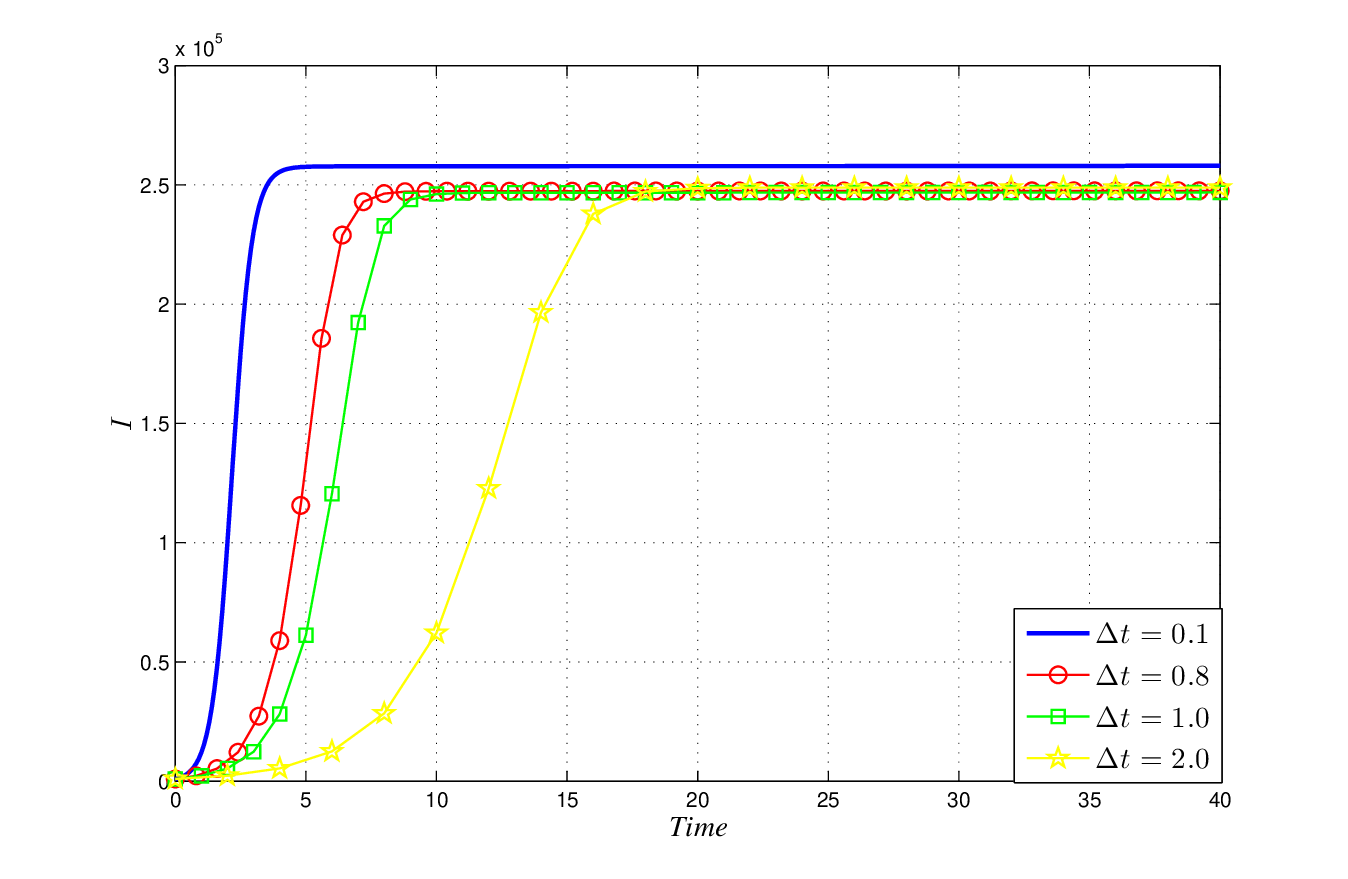}
\label{Figure:4b}
}\hfill
\caption{The numerical approximations generated by the second-order NSFD scheme with some different step sizes.}\label{Fig:4}
\end{figure}
\end{example}
\begin{example}[Global dynamics of the SIS model]
In this example, we use the proposed NSFD schemes with a small step size of
$\Delta t=10^{-4}$ to simulate the dynamics of the continuous-time SIS model over
the interval $[0,\,200]$ and thereby illustrate the global asymptotic stability results
established in Section \ref{Sec2}. To this end, we consider the parameter sets listed
in Table \ref{Table4}.
\begin{table}[H]
\centering
\caption{Parameter sets for the cases $\mathcal{R}_0<1$ and $\mathcal{R}_0>1$.}
\label{Table4}
\begin{tabular}{ccccccccc}
\hline
Set
& $\Lambda$ & $\mu$ & $\gamma$ & $\delta$
& $b$ & $\beta$ & $\mathcal{R}_0$
& Stable equilibrium\\
\hline
$1$
& $100$
& $0.02$
& $0.2$
& $0.025$
& $0.5$
& $0.1$
& $0.79$
& $E_0=(5000,\,0)$\\
$2$
& $100$
& $0.02$
& $0.2$
& $0.025$
& $0.5$
& $0.2$
& $1.59$
& $E_*=(2173.7777,\, 1256.099)$\\
\hline
\end{tabular}%
\end{table}
The numerical approximations generated by the first- and second-order NSFD schemes are depicted in Figures \ref{Fig:5}--\ref{Fig:8}. In these figures, each curve represents a numerical solution of the SIS model in the phase plane corresponding to a given initial condition, the arrows indicate the direction of evolution of the solution, and the red circle marks the globally asymptotically stable equilibrium points. It is evident that the DFE is globally asymptotically stable when $\mathcal{R}_0<1$, whereas the DEE is globally asymptotically stable when $\mathcal{R}_0>1$. These numerical observations are consistent with the theoretical results established in Section \ref{Sec2}.
\begin{figure}[H]
\subfloat[$\mathcal{R}_0 < 1$]{%
\includegraphics[height=10cm,width=15cm]{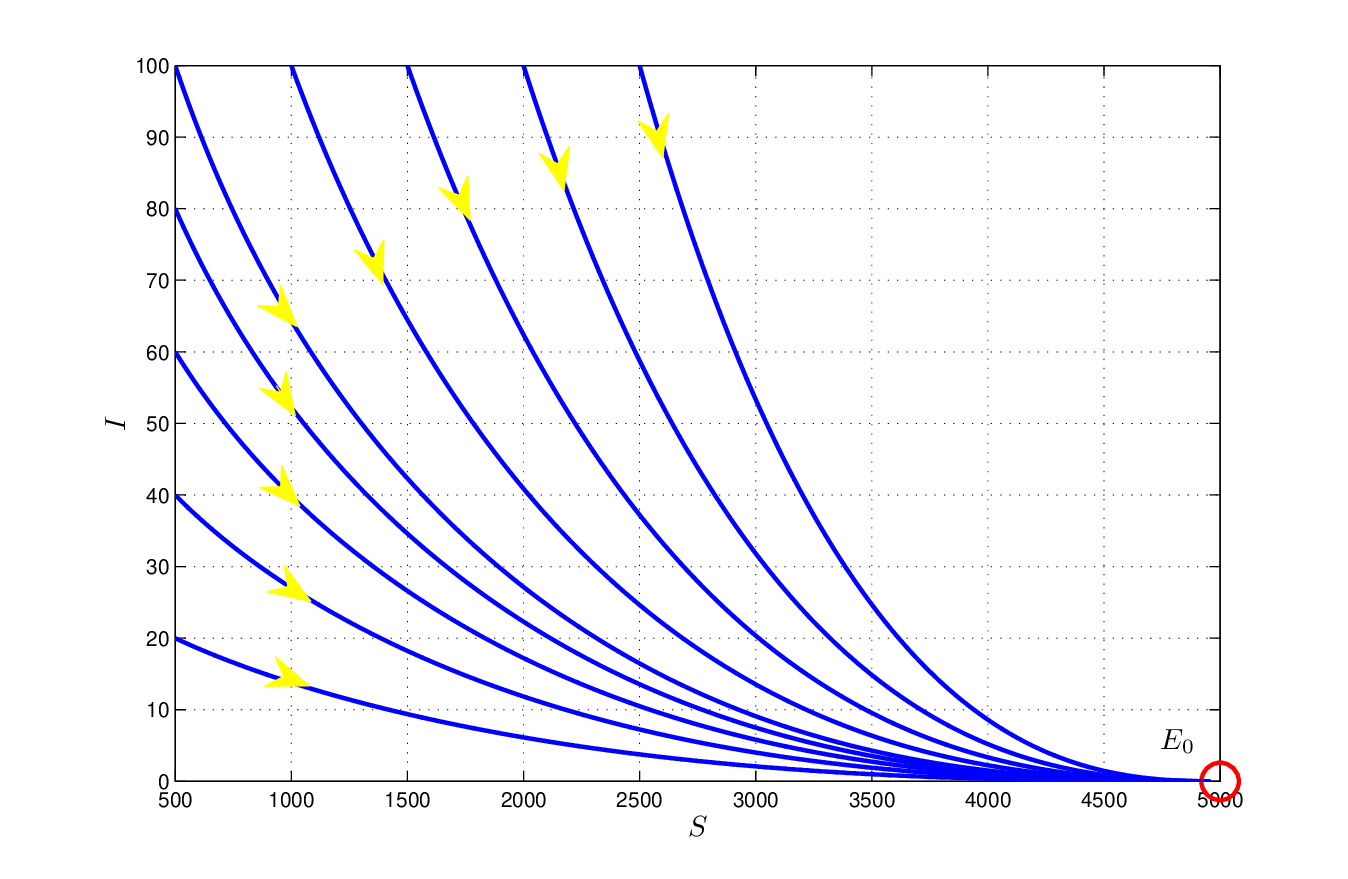}
\label{Figure:5a}
}\hfill
\subfloat[$\mathcal{R}_0 > 1$]{%
\includegraphics[height=10cm,width=15cm]{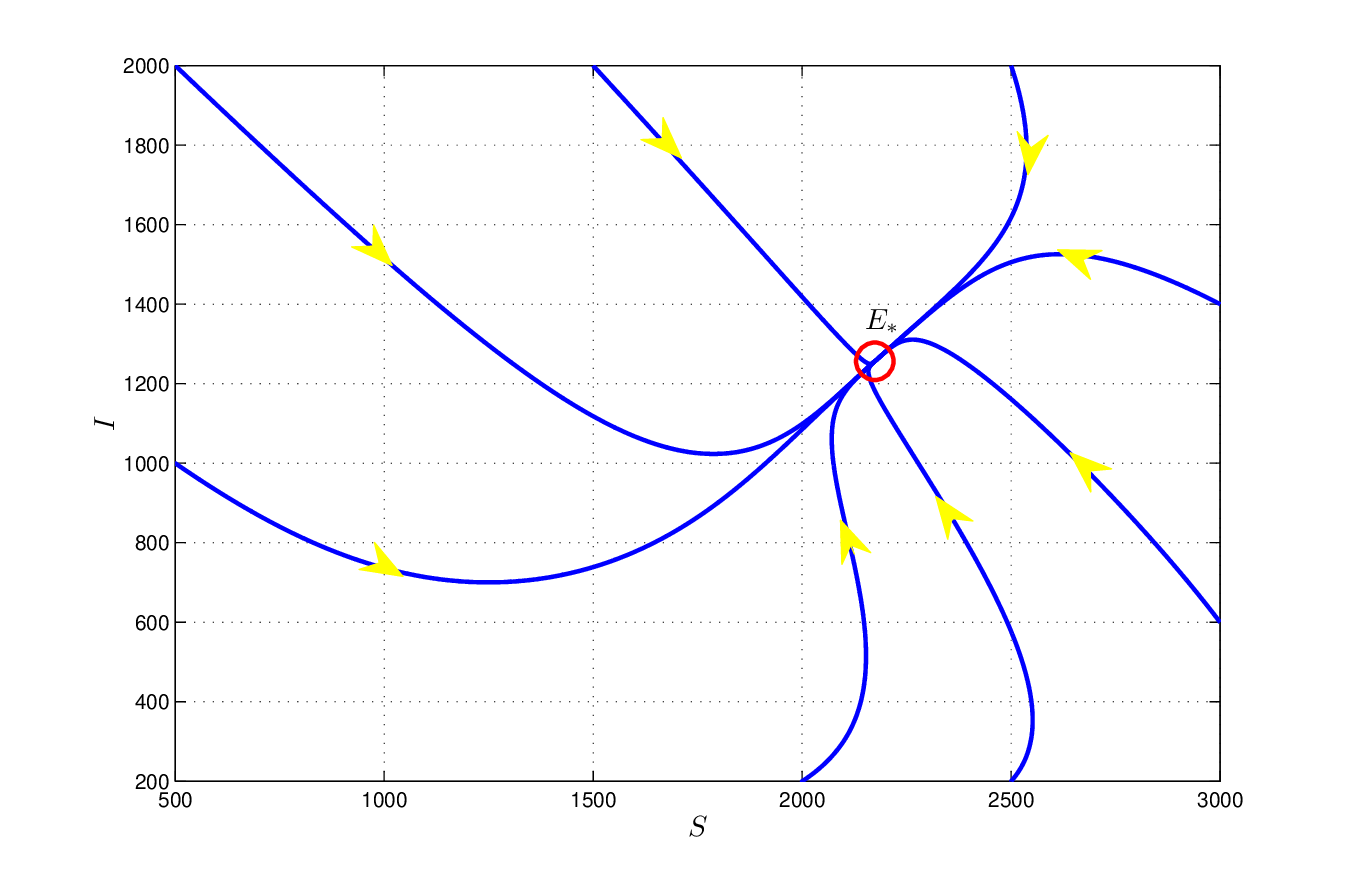}
\label{Figure:7b}
}\hfill
\caption{The numerical solutions generated by the first-order NSFD scheme.}\label{Fig:5}
\end{figure}

\begin{figure}[H]
\subfloat[$\mathcal{R}_0 < 1$]{%
\includegraphics[height=10cm,width=15cm]{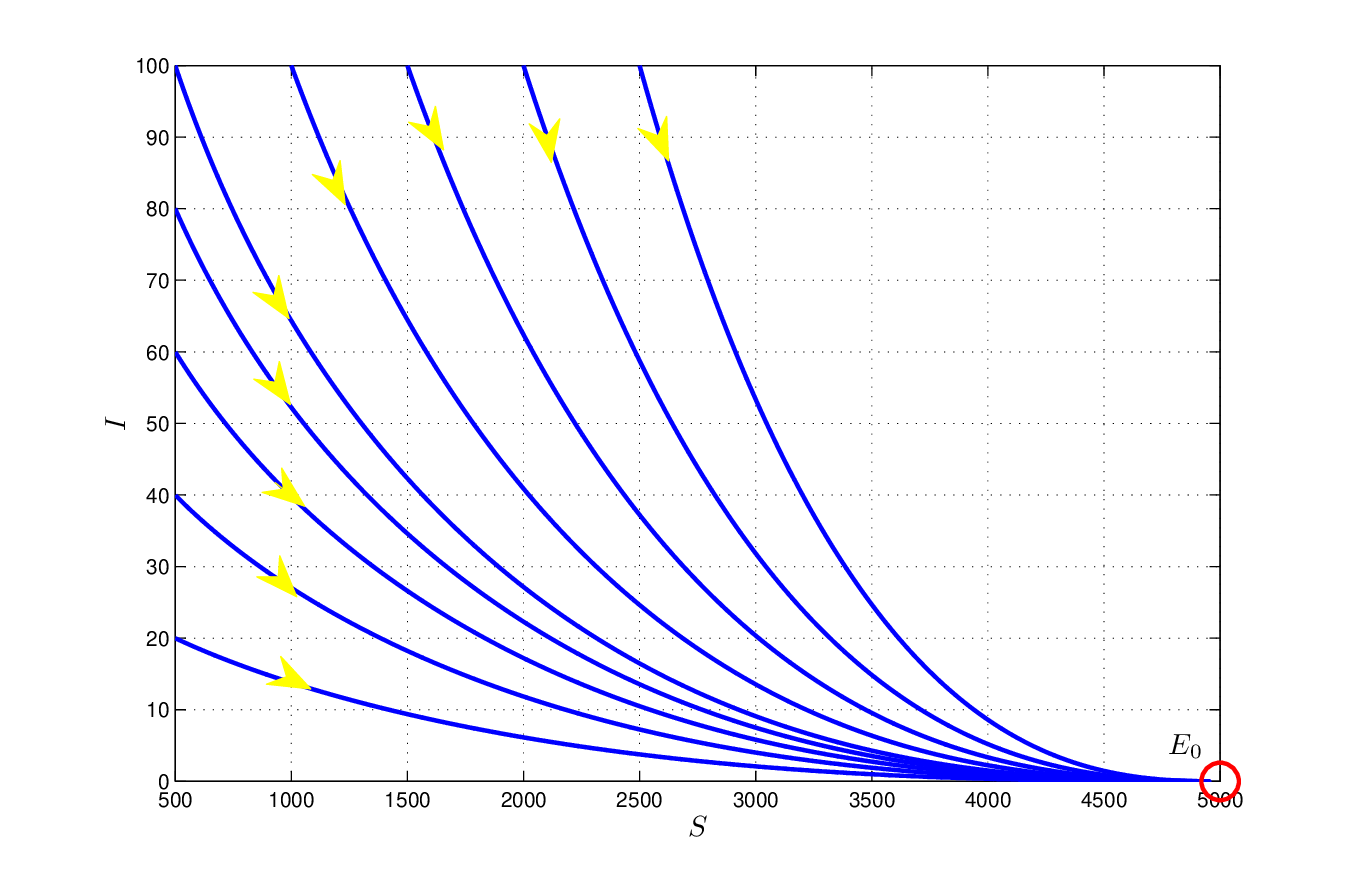}
\label{Figure:6a}
}\hfill
\subfloat[$\mathcal{R}_0 > 1$]{%
\includegraphics[height=10cm,width=15cm]{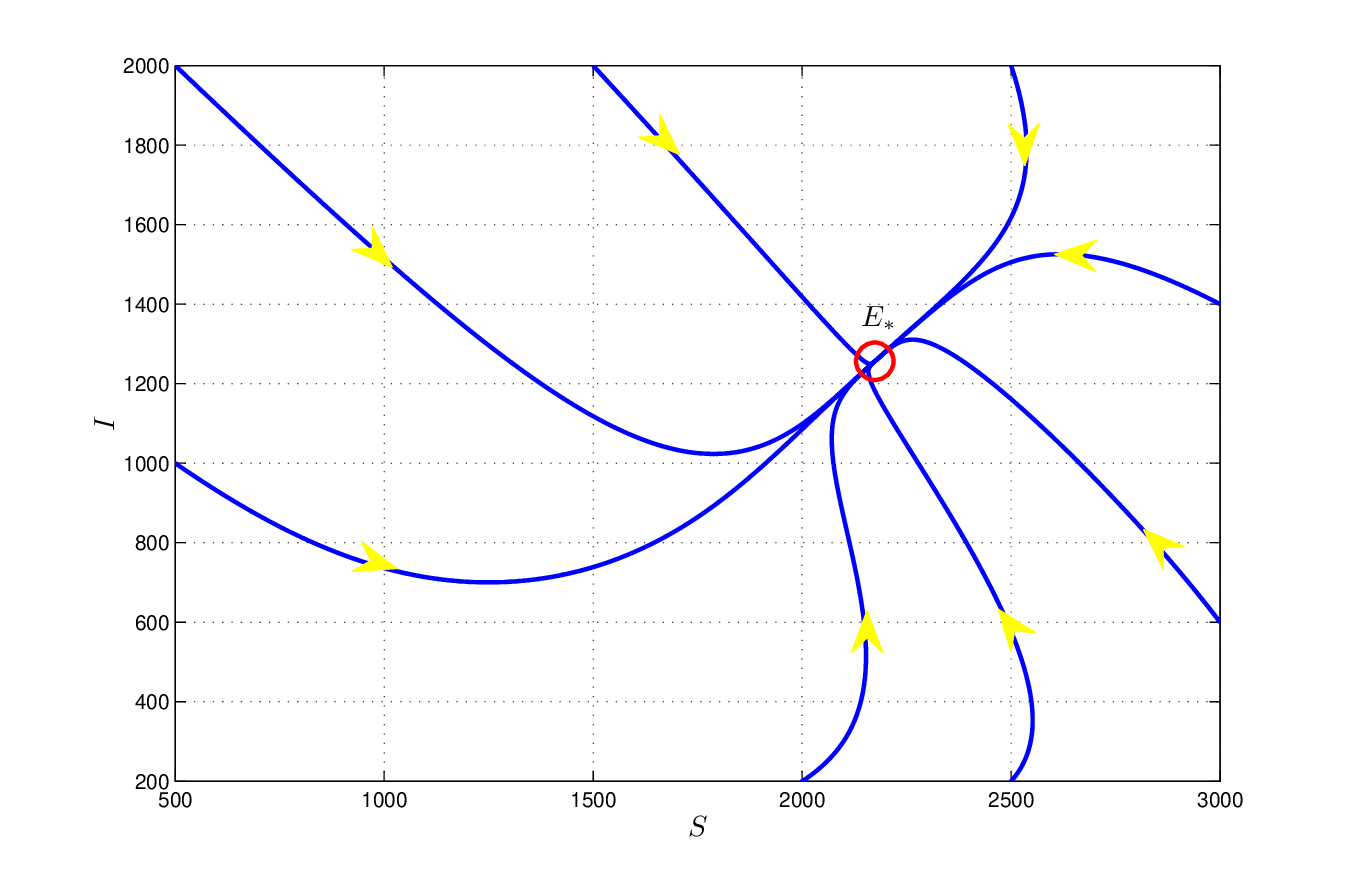}
\label{Figure:8b}
}\hfill
\caption{The numerical solutions generated by the second-order NSFD scheme.}\label{Fig:8}
\end{figure}
\end{example}
\section{Conclusions and Discussions}\label{Sec5}
In this work, we have provided a rigorous mathematical analysis of a well-known SIS epidemic model with a saturating contact rate. First, we have established the global asymptotic stability of the equilibria by employing a suitable Lyapunov function together with the Poincar\'e--Bendixson theorem and the Bendixson--Dulac criterion. The GAS results obtained in this work have improved upon previous findings for this SIS model and may provide a basis for extensions to models with more general saturating contact rates.

Second, we have constructed families of first- and second-order NSFD schemes that preserve the positivity and asymptotic stability properties of the SIS model for arbitrary step sizes. All these schemes are formulated within Mickens' framework; however, compared with their first-order counterparts, the second-order schemes employ a more elaborate construction that combines a weighted nonlocal approximation of the right-hand side with suitably renormalized denominator functions. The weighted nonlocal approximation guarantees dynamic consistency, whereas the denominator functions ensure second-order convergence.

Finally, we have conducted numerical experiments to validate the theoretical results and demonstrate the advantages of the proposed second-order NSFD schemes. The numerical results have shown good agreement with the theoretical predictions. Overall, the approach developed in this work has demonstrated its potential applicability not only to other epidemiological systems but also, more generally, to mathematical models arising from real-world applications.

In future work, we plan to extend the present approach to other epidemiological models. We also intend to develop efficient higher-order numerical methods that preserve the essential dynamical properties of the continuous--time systems under consideration.

\end{document}